\documentclass[12pt,a4paper]{article}
\usepackage{amsmath,amssymb,amsfonts}
\usepackage{amsthm}
\usepackage{mathtools}
\usepackage{marvosym}
\usepackage{authblk}
\usepackage{enumerate}
\usepackage[english]{babel}
\usepackage{enumitem}
\usepackage{todonotes}
\usepackage{indentfirst}
\usepackage{geometry}
\usepackage{graphicx}
\usepackage{tabularx}
\usepackage{tocbibind}
\usepackage{tikz}
\usepackage{array}
\usepackage{caption}
\usepackage{setspace}
\usepackage{hyperref}
\usepackage{lipsum}
\usepackage{listings}
\usepackage{multicol,multirow}
\usepackage{xcolor,xspace}
\usepackage[english]{babel}
\usepackage[export]{adjustbox}
\newcommand{\keywords}[1]{\textbf{Keywords:} #1}
\newtheorem{theorem}{Theorem}[section]
\newtheorem{corollary}{Corollary}[theorem]
\newtheorem{lemma}[theorem]{Lemma}
\newtheorem{prop}[theorem]{Proposition}

\theoremstyle{definition}
\newtheorem{definition}{Definition}[section]
\theoremstyle{remark}

\title{\textbf{On the spectrum of signed graph product and its application}}
\author[1]{Bishal Sonar\thanks{Email: bsonarnits@gmail.com}}
\author[2]{Ravi Srivastava\thanks{Corresponding author Email: ravi@nitsikkim.ac.in}}
\affil[1,2]{Department of Mathematics, National Institute of Technology Sikkim, South Sikkim 737139, India}
\date{}
\begin{document}
\parskip1ex
\parindent0pt
\maketitle

\begin{abstract}
    \noindent A  signed graph product is defined for a new product and successfully derived its adjacency spectrum, Laplacian spectrum, and signless Laplacian spectrum. Furthermore, we have generated a sequence of co-spectral signed graphs with the same spectrum and a sequence of non-co-spectral equienergetic signed graphs with the same energy but different spectra. These results represent a significant contribution to the field of graph theory and have implications for a broad range of applications.
\end{abstract}

\textbf{MSC2020 Classification:} 05C76, 05C50, 05C22\\
\keywords{$\mu$-signed graph, Adjacency spectrum, Laplacian spectrum, and signless Laplacian spectrum, Integral signed graph, equienergetic signed graphs.}

\section{Introduction}
    Signed networks offer a framework for addressing binary relationships between vertices within a network with two opposite possibilities. For instance, concepts like love and hate, trust and distrust are regarded as measures of interpersonal relationships, while alliance and antagonism between nations exemplify contradictory binary relationships on an international scale. The introduction of the signed network model for representing social systems dates back to 1956 when Harary and Cartwright proposed it as a means to generalize Heider's theory~\cite{cartwright1956structural} of balanced states in social systems, which he developed in 1946. Heider~\cite{heider1946attitudes} justified his theory of balanced states by examining potential relationships within systems comprising three entities. A signed graph is an ordered pair $\Sigma=(G,\sigma)$; where $G=(V,E)$ is a graph with vertex set $V=\{u_1,u_2,\cdots,u_n\}$, and edge set $E=\{e_1,e_2,\cdots,e_m\}$ and $\sigma:E\rightarrow \{+,-\}$ is a function which assign sign to the edges known as the signature of $\Sigma$. The sign degree of a vertex $u_i$ is denoted by $sdeg(u_i)=d^+(u_i)-d^-(u_i),$ where $d^+(u_i)$ is the number of positive edges incident to $u_i$ and $d^-(u_i)$ is the number of negative edges incident to $u_i$. Also, $d(u_i)$ is the degree of vertex $u_i$ of the underlying graph $G$. For a signed graph $\Sigma$ of order $n$, the adjacency matrix $A(\Sigma)=(a^{\sigma}_{ij})_{n\times n}$, where $a^{\sigma}_{ij}=a_{ij}\sigma(u_iu_j)$; $a_{ij}=1$ if $u_i \text{ adjacent to }u_j$ and $0$ otherwise. The diagonal degree matrix is denoted by $D(\Sigma)= d(u_i), \text{if } i=j \text{ and } 0 \text{ otherwise}$ and lastly,the Laplacian matrix is denoted by $L(\Sigma)=D(\Sigma)-A(\Sigma)$ and the signless Laplacian matrix of $\Sigma$ is denoted by $Q(\Sigma)=D(\Sigma)+A(\Sigma)$. $\mu:V\rightarrow \{+,-\}$  is a function that assign sign to the vertices of $\Sigma$ known as the marking of $\Sigma$. Here, we will discuss \textit{Canonical marking}~\cite{adhikari2023corona} defined by $\mu(u_i)=\prod_{e\in E(u_i)}\sigma(e)$; where $E(u_i)$ is the set of edges adjacent to $u_i$. The sum of absolute eigenvalues$(\lambda_i)$ of $A(\Sigma)$ is called the energy of $\Sigma$~\cite{bhat2015equienergetic} and is denoted by $$E_{\Sigma}=\sum_{i=1}^n|\lambda_i|.$$
    A graph is \textit{Integral} if all elements of the adjacency spectrum are Integer. Let $\Sigma=(G,\sigma,\mu)$ be a signed graph with marking $\mu$, then the $\mu-$signed graph~\cite{2sonar2023spectrum} is the signed graph $\Sigma_{\mu}=(G,\sigma_{\mu},\mu)$ with same marking but the following signature
        \begin{align*}
            \sigma_{\mu}(e)=\mu(u)\mu(v), \text{ for all } e(=uv)\in E(G).
        \end{align*}

     In 1970 R Frucht and F Harary~\cite{frucht1970corona} first defined the \textit{corona product} for unsigned graphs; later Cam McLeman and Erin McNicholas~\cite{mcleman2011spectra} first introduced the coronal of graphs to obtain the adjacency spectra for arbitrary graphs. Then Shu and Gui~\cite{cui2012spectrum} generalized it and defined the coronal of the Laplacian and the signless Laplacian matrix of unsigned graphs. SP Joseph~\cite{joseph2023graph} recently introduced a new product and worked on its adjacency spectra and its application in generating non-cospectral equienergetic graphs, however, no work was done for signed graphs, we have calculated the spectrum of adjacency, Laplacian, and signless Laplacian matrix using the definition of sign coronal by Amrik et al.~\cite{adhikari2023corona}. As an application we have applied it to generate the condition for getting the integral graph and in generating non-cospectral equienergetic signed graphs. The graphs considered in this paper are simple, undirected, and finite.

    \section{\textbf{Preliminaries}}
       
    \begin{theorem}~\cite{harary1953notion} \label{Th1}
        A signed graph $\Sigma=(G,\sigma,\mu)$ is balanced if and only if there exists a marking $\mu$ such that for each edge $u_iu_j$ in $\Sigma$ one has $\sigma(u_iu_j)=\mu(u_i)\mu(u_j).$
    \end{theorem}

   \begin{lemma}~\cite{2sonar2023spectrum}\label{L1}
    For a signed graph $\Sigma=(G,\sigma,\mu)$ the $\mu-$signed graph $\Sigma_\mu=(G,\sigma_\mu,\mu)$ is always balanced.
   \end{lemma}
    
    \begin{lemma}\label{L2}~\cite{bapat2010graphs}
        Let $S_1,$ $S_2,$ $S_3,$ and $S_4$ be matrix of order $m_1\times m_1, m_1\times m_2, m_2\times m_1, m_2\times m_2$ respectively with $S_1$ and $S_4$ are invertible. Then 
            \begin{align*}
                \det\begin{bmatrix}
                S_1&S_2\\S_3&S_4
        \end{bmatrix}&=\det(S_1)\det(S_4-S_3S_1^{-1}S_2)\\
        &=\det(S_4)\det(S_1-S_2S_4^{-1}S_3).
        \end{align*}
    \end{lemma}

        Let $P=(p_{ij})_{n_1\times m_1}$ and $Q=(q_{ij})_{n_2\times m_2}$ be two matrices, then the Kronecker product~\cite{neumaier1992horn} $P\otimes Q$ of matrix $P$ and $Q$ is a $n_1n_2\times m_1m_2$ matrix formed by replacing each $p_{ij}$ by $p_{ij}Q$. Kronecker product is associative, $(P\otimes Q)^T=P^T\otimes Q^T$, $(P\otimes Q)(R \otimes S)=PR\otimes QS$, given the product $PR$ and $QS$ exists, $(P\otimes Q)^{-1}=P^{-1}\otimes Q^{-1}$, for non-singular matrices $P$ and $Q$, and if $P$ and $Q$ are $n\times n$ and $m\times m$ matrices, then $\det(P\times Q)=(\det P)^m(\det Q)^n$. In the next section, we define the new graph product for signed graph.

    \section{Spectral properties}
        
        \subsection{Definition}~\label{Def1}
        Let $\Sigma_1=(S_1,\sigma_1,\mu_1)$ be a signed graph of size $e_1$ and order $n_1$ with a set of vertex $\{u_1,u_2,\cdots,u_{n_1}\}$ and $\Sigma_2=(S_2,\sigma_2,\mu_2)$ be a signed graph of size $e_2$ and order $n_2$ with vertex set $\{v_1,v_2,\cdots,v_{n_2}\}$. Then the signed graph product is defined as follows:
        \begin{enumerate}[label=(\roman*)]
            \item The vertex set of $\Sigma_1\circledast \Sigma_2$ is given by\\
           $\{a_{11},a_{12},\cdots,a_{1n_2},a_{21},a_{22},\cdots,a_{2n_2},\cdots,a_{n_11},a_{n_12},\cdots,a_{n_1n_2},b_{11},b_{12},\cdot,b_{1n_2},b_{21},b_{22},\cdots, \\ b_{2n_2},\cdots,b_{n_11},b_{n_12},\cdots,b_{n_1n_2}\}$, with $\mu(a_{ij})=u_i;$ for $1\leq i\leq n_1$, and $\mu(b_{ij}=v_j);$ for $1\leq i\leq n_1$.
           \item The set of edges of $\Sigma_1\circledast \Sigma_2$ consists of the following three types of edges:\\
           \begin{itemize}\vspace{-.5cm}
               \item If the edge $(u_i,u_j)\in E(\Sigma_1)$, then the edges $(a_{ik},a_{jl});$ for $1\leq k,l\leq n_2$ belong to $\Sigma_1\circledast \Sigma_2$. Here $\sigma(a_{ik},a_{jl})=\mu(a_{ik})\cdot \mu(a_{jl}).$
               \item If the edge $(v_i,v_j)\in E(\Sigma_2)$, then the edges $(b_{ri},b_{rj});$ for $1\leq r\leq n_1$ belongs to $\Sigma_1\circledast \Sigma_2$. Here $\sigma(b_{ri},b_{rj})=\mu(b_{ri})\cdot \mu(b_{rj}).$
               \item For $1\leq i\leq n_1$, the edges $(a_{ip},b_{iq});$ for $1\leq p,q \leq n_2$ belongs to $\Sigma_1\circledast \Sigma_2$. Here $\sigma(a_{ip},b_{iq})=\mu(a_{ip})\cdot \mu(b_{iq}).$
           \end{itemize}
        \end{enumerate}

        \begin{figure}
            \centering
            \includegraphics{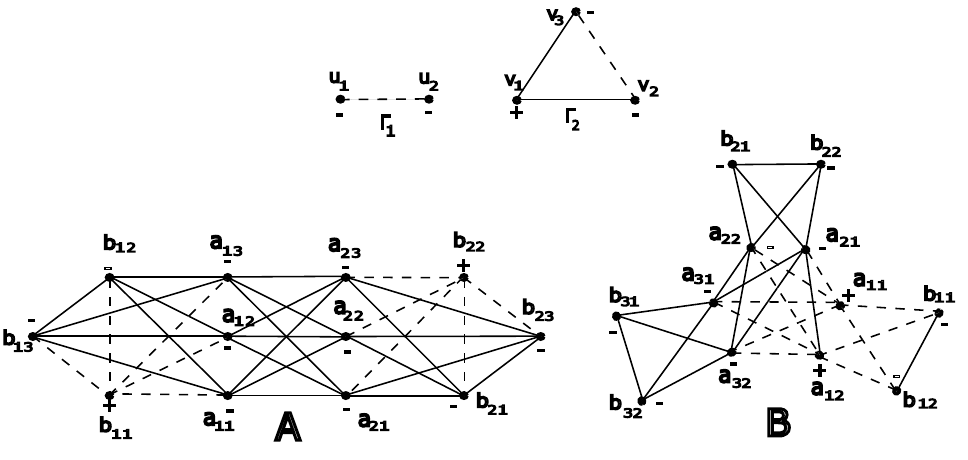}
            \caption{$A=\Sigma_1\circledast\Sigma_2$ and $B=\Sigma_2\circledast\Sigma_1.$ Doted lines implies negative edge and Straight line implies a positive edge.}
            \label{F1}
        \end{figure}
        Clearly, there are $2n_1n_2$ vertices in the product signed graph and $n_2^2e_1, n_1e_2,$ and $n_1n_2^2$ number of edges of first, second, and third type respectively. So, the total number of edges in the product signed graph is $n_2^2(n_1+e_1)+n_1e_2.$\\
        This graph product is not commutative, as already proved in \cite{joseph2023graph}. We illustrate the above construction of signed graph product using simple example in Figure~\ref{F1}. It can be verified that for $n_2=1$ the given product is the same as the corona product.
        
    \begin{definition}\cite{singh2023structural}
        Let $\Sigma=(G,\sigma,\mu)$ be a signed graph of order $n$ and $N$ be the signed graph matrix of $\Sigma$. Now, $N$ being a matrix over the rational function field $\mathbb{C}(x)$, the matrix $xI_n-N$ is non-singular and so it is invertible. The signed $N$-coronal $\chi_N(x) \in \mathbb{C}(x)$ of $\Sigma$ is defined as,
        \begin{equation}\label{Eqn1}
        \chi_N(x)=\mu(\Sigma)^T(xI_n-N)^{-1}\mu(\Sigma).
        \end{equation}
    \end{definition}
        
    Replacing $N$ in (\ref{Eqn1}) by adjacency, Laplacian, and signless Laplacian matrix we get the respective coronals.\\
    
        Using the above notions for two signed graphs $\Sigma_1$ and $\Sigma_2$ with $n_1$ and $n_2$ vertices, we have the following;\vspace{1pt}
   \begin{equation}\label{Eqn2}
        \phi(\Sigma_1)\otimes (1_{n_2}\mu(\Sigma_2)^T)\big(x I_{n_1n_2}-I_{n_1}\otimes A(\Sigma_{2\mu})\big)^{-1} \phi(\Sigma_1)\otimes (\mu(\Sigma_2)1_{n_2}^T)=-I_{n_1}\otimes\chi_{A(\Sigma_{2\mu})}(x)J_{n_2} \vspace{2pt}
    \end{equation}
    \begin{equation}\label{Eqn3}
    \begin{split}
        &\phi(\Sigma_1)\otimes (1_{n_2}\mu(\Sigma_2)^T)\big((x-r_1-n_2)I_{n_1n_2}+I_{n_1}\otimes A(\Sigma_{2\mu})\big)^{-1} \phi(\Sigma_1)\otimes (\mu(\Sigma_2)1_{n_2}^T)\\
        &=-I_{n_1}\otimes\chi_{A(\Sigma_{2\mu})}(x-n_2)J_{n_2} \vspace{2pt}
        \end{split}
    \end{equation}

        Clearly, for a signed graph $\Sigma$ of order $n$, and a signed graph matrix $N$,
        \begin{equation}\label{Eqn4}
        \begin{split}
            \chi_{N}(x)&=\frac{\mu(\Sigma)^TAdj(x I_n-N)\mu(\Sigma)}{\det(x I_n-N)}\\
            &=\frac{p(N,x)}{f(N,x)},
        \end{split}
        \end{equation}

        where $f(N,x)$ is the characteristics polynomial of the matrix $N$ of the signed graph $\Sigma$ and $p(N,x)$ is a polynomial of degree $n-1$. The aforementioned polynomial ratio can be further reduced if the greatest common divisor of these polynomials is not a constant
        \begin{equation}\label{Eqn5}
            \chi_{N}(x)=\frac{P_{d-1}(x)}{F_d(x)},
        \end{equation}
        where $P_{d-1}(x)$ and $F_d(x)$ are polynomials of degree $d-1$ and $d$ respectively and $$gcd\big(p(N,x),f(N,x)\big)=R_{n-d}(x),$$
        is of degree $n-d$.
        
    \subsection{A-spectra}
        \begin{theorem}~\label{Th2}
            Let $i=1,2;$ $\Sigma_i$ be two signed graphs of order $n_i$ and eigenvalues $\lambda_{i1},\lambda_{i2},\cdots,\lambda_{in_i}$. Let $\Sigma_{i\mu}$ be the $\mu$-signed graph of $\Sigma_i$, and ${\lambda'}_{i1},{\lambda'}_{i2},\cdots,{\lambda'}_{in_i}$ be the corresponding eigenvalues. Then the adjacency characteristics polynomial of $\Sigma_1\circledast\Sigma_2$ is,\\
            $$f(A(\Sigma_1\circledast\Sigma_2),x)=x^{n_1(n_2-1)}{R'}_{n_2-d}(x)^{n_1}\prod_{i=1}^{n_1}[x{F'}_d(x)-n_2({\lambda'}_{1i}{F'}_d(x)+{P'}_{d-1}(x))]$$
            where ${R'}_{n_2-d}, {F'}_d, {P'}_{d-1}$ are given by equation~(\ref{Eqn5}) for $\mu$-signed graph.
        \end{theorem}

        \begin{proof}
            Following the definition~\ref{Def1}, the adjacency matrix of $\Sigma_1\circledast\Sigma_2$ can be written as,
            $$A(\Sigma_1\circledast\Sigma_2)=\begin{bmatrix}
                A(\Sigma_{1\mu})\otimes J_{n_2} &\phi(\Sigma_1)\otimes (1_{n_2}\mu(\Sigma_2)^T) \\
                \phi(\Sigma_1)\otimes (\mu(\Sigma_2)1_{n_2}^T) &I_{n_1}\otimes A(\Sigma_{2\mu})
            \end{bmatrix}.$$
            Then using Lemma~\ref{L2} the adjacency polynomial of $\Sigma_1\circledast\Sigma_2$ is,\\
            \begin{align}
                f(A(\Sigma_1\circledast\Sigma_2),x)&=\det\big[xI_{n_1(2+n_2)}-A(\Sigma_1\circledast\Sigma_2)\big]\notag\\
                &=\det\begin{bmatrix}
                    xI_{n_1n_2}-A(\Sigma_{1\mu})\otimes J_{n_2} &-\phi(\Sigma_1)\otimes (1_{n_2}\mu(\Sigma_2)^T)\notag\\
                -\phi(\Sigma_1)\otimes (\mu(\Sigma_2)1_{n_2}^T) &xI_{n_1n_2}-I_{n_1}\otimes A(\Sigma_{2\mu})
                \end{bmatrix} \notag\\
                &=\det\big[xI_{n_1n_2}-I_{n_1}\otimes A(\Sigma_{2\mu})\big] \cdot\det\Big[{xI_{n_1n_2}-A(\Sigma_{1\mu})\otimes J_{n_2}}-\notag\\
                &\{\phi(\Sigma_1)\otimes (1_{n_2}\mu(\Sigma_2)^T)\}\{xI_{n_1n_2}-I_{n_1}\otimes A(\Sigma_{2\mu})\}^{-1}\{\phi(\Sigma_1)\otimes (\mu(\Sigma_2)1_{n_2}^T)\}\Big]\notag\\
                &=\det(I_{n_1})^{n_2}\cdot\det\big[xI_{n_2}-A(\Sigma_{2\mu})\big]^{n_1}\cdot \det\Big[{xI_{n_1n_2}-A(\Sigma_{1\mu})\otimes J_{n_2}}- \{I_{n_1}\otimes\notag\\
                &\chi_{A(\Sigma_{2\mu})}(x)J_{n_2}\}\big]\notag\\
                &=\det(xI_{n_2}-A(\Sigma_{2\mu}))^{n_1}\cdot \det\Big[{xI_{n_1n_2}-A(\Sigma_{1\mu})\otimes J_{n_2}}- \{I_{n_1}\otimes \chi_{A(\Sigma_{2\mu})}(x)J_{n_2}\}\big]\notag\\
                &=f(A(\Sigma_{2\mu}),x)^{n_1}\cdot\prod(x-\alpha_i\beta_j), \text{where the product is taken over all the eigen-}\notag\\
                &\text{values ${\alpha_i}'s$ and ${\beta_j}'s$ of $A(\Sigma_{1\mu})+\chi_{A(\Sigma_{2\mu})}(x)I_{n_1}$ and $J_{n_2}$ respectively.}\notag\\
                &=x^{n_1(n_2-1)}f(A(\Sigma_{2\mu}),x)^{n_1}\cdot \prod_{i=1}^{n_1}\big[x-n_2({\lambda'}_{1i}+\chi_{A(\Sigma_{2\mu})}(x))\big]\label{Eqn6}\\ 
                &=x^{n_1(n_2-1)}\big(R'_{n_2-d}(x)F'_d(x)\big)^{n_1}\prod_{i=1}^{n_1}\Big[x-n_2({\lambda'}_{1i}+\frac{P'_{d-1}(x)}{F'_d(x)}\Big]\notag\\
                &=x^{n_1(n_2-1)}\big(R'_{n_2-d}(x)\big)^{n_1}\prod_{i=1}^{n_1}\Big[xF'_d(x)-n_2({\lambda'}_{1i}F'_d(x)+P'_{d-1}(x))\Big]. \notag
          \end{align}

          $\therefore$ ,$f(A(\Sigma_1\circledast\Sigma_2),x)=x^{n_1(n_2-1)}{R'}_{n_2-d}(x)^{n_1}\prod_{i=1}^{n_1}[x{F'}_d(x)-n_2({\lambda'}_{1i}{F'}_d(x)+{P'}_{d-1}(x))]$\\
          To get this result we have used different results of Kronecker Product mentioned in the Preliminary section.
        \end{proof}
        \begin{corollary}\label{C1}
            For $n_2=1$, the adjacency characteristics polynomial is the same as the Adjacency characteristics polynomial of the corona product of two graphs.
        \end{corollary}
        \begin{proof}
            From equation (\ref{Eqn6}) we have;\\ 
            \begin{align*}
                f(A(\Sigma_1\circledast\Sigma_2),x)&=f(A(\Sigma_{2\mu}),x)^{n_1}\cdot \prod_{i=1}^{n_1}\big[x-({\lambda'}_{1i}+\chi_{A(\Sigma_{2\mu})}(x))\big]\\
                &=f(A(\Sigma_{2\mu}),x)^{n_1}\cdot f(A(\Sigma_{1\mu})-\chi_{A(\Sigma_{2\mu})}(x)).
            \end{align*}
            Now, from Lemma~\ref{L1} we have the $\mu$ signed graph of a signed graph is always balanced. So, $A(\Sigma_{\mu})$ and $A(\Sigma)$ has the same eigenvalues.\\
            Hence, $f(A(\Sigma_1\circledast\Sigma_2),x)=f(A(\Sigma_2),x)^{n_1}\cdot f(A(\Sigma_1)-\chi_{A(\Sigma_2)}(x))=$ Corona product of $\Sigma_1$ and $\Sigma_2.$
        \end{proof}

        \begin{corollary}
         Let $\Sigma_1$ and $\Sigma_2$ be two $A(\Sigma_{\mu})-$co-spectral signed graphs and $\Sigma$ is any arbitrary signed graph, then 
            \begin{enumerate}[label=(\roman*)]
                \item $\Sigma_1\circledast\Sigma$ and $\Sigma_2\circledast\Sigma$ are $A-$co-spectral.
                \item $\Sigma\circledast\Sigma_1$ and $\Sigma\circledast\Sigma_2$ are co-spectral if $\chi_{A(\Sigma_{1\mu})}(x)=\chi_{A(\Sigma_{2\mu})}(x).$
            \end{enumerate}
     \end{corollary}
    \subsection{L-Spectra and Q-spectra of regular signed graphs}
     
        \begin{theorem}~\label{Th3}
            Let $i=1,2;$ $\Sigma_i$ be two $r_i$-regular signed graphs of order $n_i$ and Laplacian eigenvalues $\mu_{i1},\mu_{i2},\ldots,\mu_{in_i}$. Let $\Sigma_{i\mu}$ be the $\mu$-signed graph of $\Sigma_i$, and ${\mu'}_{i1},{\mu'}_{i2},\ldots,{\mu'}_{in_i}$ be the corresponding Laplacian eigenvalues. Then the Laplacian characteristic polynomial of $\Sigma_1\circledast\Sigma_2$ is\\
            $f(L(\Sigma_1\circledast\Sigma_2),x)=(x-r_1n_2-n_2))^{n_1(n_2-1)}\big(R'_{n_2-d}(x-n_2)\big)^{n_1}\prod_{i=1}^{n_1}\Big[(x-n_2)F'_d(x-n_2)-n_2({\mu'}_{1i}F'_d(x-n_2)+P'_{d-1}(x-n_2)\Big],$\\
            where ${R'}_{n_2-d}, {F'}_d, {P'}_{d-1}$ are given by equation~(\ref{Eqn5}) of $\mu$-signed graph for Laplacian matrix.
        \end{theorem}
        \begin{proof}
            Following the definition~\ref{Def1} of $\Sigma_1\circledast\Sigma_2$ we have
            $$A(\Sigma_1\circledast\Sigma_2)=\begin{bmatrix}
                A(\Sigma_{1\mu})\otimes J_{n_2} &\phi(\Sigma_1)\otimes (1_{n_2}\mu(\Sigma_2)^T) \\
                \phi(\Sigma_1)\otimes (\mu(\Sigma_2)1_{n_2}^T) &I_{n_1}\otimes A(\Sigma_{2\mu})
            \end{bmatrix}$$
            $$D(\Sigma_1\circledast\Sigma_2)=diag\big(n_2(r_1+1)I_{n_1n_2},(r_2+n_2)I_{n_1n_2}\big)$$
            Then the Laplacian matrix of $\Sigma_1\circledast\Sigma_2$ is
            $$L(\Sigma_1\circledast\Sigma_2)=\begin{bmatrix}
                n_2(r_1+1)I_{n_1n_2}-A(\Sigma_{1\mu})\otimes J_{n_2} &-\phi(\Sigma_1)\otimes (1_{n_2}\mu(\Sigma_2)^T) \\
                -\phi(\Sigma_1)\otimes (\mu(\Sigma_2)1_{n_2}^T) &I_{n_1}\otimes (n_2I_{n_2}+L(\Sigma_{2\mu}))
            \end{bmatrix}$$
            Then using Lemma~\ref{L2} the Laplacian polynomial of $\Sigma_1\circledast\Sigma_2$ is,\\
            \begin{align}
                &f(L(\Sigma_1\circledast\Sigma_2),x)=\det\big[xI_{2n_1n_2}-L(\Sigma_1\circledast\Sigma_2)\big]\notag\\
                &=\det\begin{bmatrix}
                    (x-n_2(r_1+1))I_{n_1n_2}+A(\Sigma_{1\mu})\otimes J_{n_2} &\phi(\Sigma_1)\otimes (1_{n_2}\mu(\Sigma_2)^T)\notag\\
                \phi(\Sigma_1)\otimes (\mu(\Sigma_2)1_{n_2}^T) &I_{n_1}\otimes (x-n_2)I_{n_2}-L(\Sigma_{2\mu})
                \end{bmatrix} \notag\\
                &=\det\big[I_{n_1}\otimes (x-n_2)I_{n_2}-L(\Sigma_{2\mu})\big] \cdot\det\Big[\{(x-n_2(r_1+1))I_{n_1n_2}+A(\Sigma_{1\mu})\notag\otimes J_{n_2}\}-\{\phi(\Sigma_1)\otimes\\
                & (1_{n_2}\mu(\Sigma_2)^T)\}\{(I_{n_1}\otimes (x-n_2)I_{n_2}-L(\Sigma_{2\mu})\}^{-1}\{\phi(\Sigma_1)\notag\otimes (\mu(\Sigma_2)1_{n_2}^T)\}\Big]\notag\\
                &=\det(I_{n_1})^{n_2}\cdot\det\big[(x-n_2)I_{n_2}-L(\Sigma_{2\mu})\big]^{n_1}\cdot \det\Big[\{(x-n_2(r_1+1))I_{n_1n_2}+A(\Sigma_{1\mu})\otimes J_{n_2}\}+\notag\\
                &\{I_{n_1}\otimes\chi_{L(\Sigma_{2\mu})}(x-n_2)J_{n_2}\}\big]\notag\\
                &=\det((x-n_2)I_{n_2}-L(\Sigma_{2\mu}))^{n_1}\cdot \det\Big[\{(x-n_2(r_1+1))I_{n_1n_2}+A(\Sigma_{1\mu})
                \otimes J_{n_2}\}-\{I_{n_1}\otimes\notag\\
                & \chi_{L(\Sigma_{2\mu})}(x-n_2)J_{n_2}\}\big]\notag\\
                &=f(L(\Sigma_{2\mu}),(x-n_2))^{n_1}\cdot\det\big[(x-n_2r_1-n_2)I_{n_1n_2}+A(\Sigma_{1\mu})-\chi_{L(\Sigma_{1\mu})}(x-n_2)I_{n_1}\otimes J_{n_2}\big]\notag\\
                &=f(L(\Sigma_{2\mu}),(x-n_2))^{n_1}\cdot\prod((x-n_2(r_1+1))-\alpha_i\beta_j), \text{here the product is over the eigenvalues}\notag\\
                &\alpha_i \text{ and } \beta_j \text{ of } A(\Sigma_{1\mu})-\chi_{L(\Sigma_{2\mu})}(x-n_2)I_{n_1}\text{ and } J_{n_2}\text{ respectively.}\notag
                \end{align}
                \begin{align}
                &=(x-r_1n_2-n_2)^{n_1(n_2-1)}f(L(\Sigma_{2\mu}),(x-n_2))^{n_1}\cdot \prod_{i=1}^{n_1}\big[(x-n_2-n_2r_1)+n_2({\lambda'}_{1i}-\chi_{L(\Sigma_{2\mu})}(x-n_2))\big]\notag\\
                &=(x-r_1n_2-n_2)^{n_1(n_2-1)}f(L(\Sigma_{2\mu}),(x-n_2))^{n_1}\cdot \prod_{i=1}^{n_1}\big[(x-n_2)-n_2{\mu'}_{1i}-n_2\chi_{L(\Sigma_{2\mu})}(x-n_2))\big]\label{Eqn7}\\
                &=(x-r_1n_2-n_2))^{n_1(n_2-1)}\big(R'_{n_2-d}(x-n_2)F'_d(x-n_2)\big)^{n_1}\prod_{i=1}^{n_1}\Big[(x-n_2)-n_2({\mu'}_{1i}+\frac{P'_{d-1}(x-n_2)}{F'_d(x-n_2)}\Big]\notag\\
                &=(x-r_1n_2-n_2))^{n_1(n_2-1)}\big(R'_{n_2-d}(x-n_2)\big)^{n_1}\prod_{i=1}^{n_1}\Big[(x-n_2)F'_d(x-n_2)-n_2({\mu'}_{1i}F'_d(x-n_2)+\notag\\
                &P'_{d-1}(x-n_2)\Big]. \notag
            \end{align}

          Therefore ,$f(L(\Sigma_1\circledast\Sigma_2),x)=(x-r_1n_2-n_2))^{n_1(n_2-1)}\big(R'_{n_2-d}(x-n_2)\big)^{n_1}\prod_{i=1}^{n_1}\Big[(x-n_2)F'_d(x-n_2)-n_2({\mu'}_{1i}F'_d(x-n_2)+P'_{d-1}(x-n_2)\Big].$\\
          To get this result we have used different results of Kronecker Product mentioned in the Preliminary section.
        \end{proof}

        \begin{corollary}
            For $n_2=1$, the Laplacian characteristics polynomial is the same as the Laplacian characteristics polynomial of the corona product of two graphs.
        \end{corollary}
        \begin{proof}
            From equation (\ref{Eqn7}) we have;
            \begin{equation*}
                \begin{split}
                    f(L(\Sigma_1\circledast\Sigma_2),x)&=(x-r_1n_2-n_2)^{n_1(n_2-1)}f(L(\Sigma_{2\mu}),(x-n_2))^{n_1}\cdot \prod_{i=1}^{n_1}\big[(x-n_2)-n_2{\mu'}_{1i}-\notag\\&n_2\chi_{L(\Sigma_{2\mu})}(x-n_2))\big]\notag\\
                    &=f(L(\Sigma_{2\mu}),(x-1))^{n_1}\cdot \prod_{i=1}^{n_1}\big[(x-1)-{\mu'}_{1i}-\chi_{L(\Sigma_{2\mu})}(x-1)).
                \end{split}
            \end{equation*}
            Now, from Lemma~\ref{L1} we have the $\mu$ signed graph of a signed graph is always balanced. So, $L(\Sigma_{\mu})$ and $L(\Sigma)$ has the same eigenvalues.\\
            Hence, $f(L(\Sigma_1\circledast\Sigma_2),x)=f(L(\Sigma_{2}),(x-1))^{n_1}\cdot \prod_{i=1}^{n_1}\big[(x-1)-{\mu}_{1i}-\chi_{L(\Sigma_{2})}(x-1))=$ Laplacian characteristics polynomial of corona product of $\Sigma_1$ and $\Sigma_2.$ 
        \end{proof}

        \begin{corollary}
         Let $\Sigma_1$ and $\Sigma_2$ be two $L(\Sigma_{\mu})$-co-spectral signed graphs and $\Sigma$ is any arbitrary signed graph, then 
            \begin{enumerate}[label=(\roman*)]
                \item $\Sigma_1\circledast\Sigma$ and $\Sigma_2\circledast\Sigma$ are $L$-co-spectral.
                \item $\Sigma\circledast\Sigma_1$ and $\Sigma\circledast\Sigma_2$ are $L$-co-spectral if $\chi_{L(\Sigma_{1\mu})}(x)=\chi_{L(\Sigma_{2\mu})}(x).$
            \end{enumerate}
     \end{corollary}

        \begin{theorem}
            Let $i=1,2;$ $\Sigma_i$ be two $r_i$-regular signed graphs of order $n_i$ and signless Laplacian eigenvalues $\nu_{i1},\nu_{i2},\ldots,\nu_{in_i}$. Let $\Sigma_{i\mu}$ be the $\mu$-signed graph of $\Sigma_i$, and ${\nu'}_{i1},{\nu'}_{i2},\ldots,{\nu'}_{in_i}$ be the corresponding signless Laplacian eigenvalues. Then the Laplacian characteristic polynomial of $\Sigma_1\circledast\Sigma_2$ is\\
            $f(Q(\Sigma_1\circledast\Sigma_2),x)=(x-r_1n_2-n_2))^{n_1(n_2-1)}\big(R'_{n_2-d}(x-n_2)\big)^{n_1}\prod_{i=1}^{n_1}\Big[(x-n_2)F'_d(x-n_2)-n_2({\nu'}_{1i}F'_d(x-n_2)+P'_{d-1}(x-n_2)\Big],$\\
            where ${R'}_{n_2-d}, {F'}_d, {P'}_{d-1}$ are given by equation~(\ref{Eqn5}) of $\mu$-signed graph for signless Laplacian matrix.
        \end{theorem}
        \begin{proof}
             Following the definition~\ref{Def1} of $\Sigma_1\circledast\Sigma_2$ we have
            $$A(\Sigma_1\circledast\Sigma_2)=\begin{bmatrix}
                A(\Sigma_{1\mu})\otimes J_{n_2} &\phi(\Sigma_1)\otimes (1_{n_2}\mu(\Sigma_2)^T) \\
                \phi(\Sigma_1)\otimes (\mu(\Sigma_2)1_{n_2}^T) &I_{n_1}\otimes A(\Sigma_{2\mu})
            \end{bmatrix}$$
            $$D(\Sigma_1\circledast\Sigma_2)=diag\big(n_2(r_1+1)I_{n_1n_2},(r_2+n_2)I_{n_1n_2}\big)$$
            Then the signless Laplacian matrix of $\Sigma_1\circledast\Sigma_2$ is
            $$Q(\Sigma_1\circledast\Sigma_2)=\begin{bmatrix}
                n_2(r_1+1)I_{n_1n_2}+A(\Sigma_{1\mu})\otimes J_{n_2} &\phi(\Sigma_1)\otimes (1_{n_2}\mu(\Sigma_2)^T) \\
                \phi(\Sigma_1)\otimes (\mu(\Sigma_2)1_{n_2}^T) &I_{n_1}\otimes (n_2I_{n_2}+Q(\Sigma_{2\mu}))
            \end{bmatrix}$$
            The later part of proof is same as theorem~\ref{Th3}.
        \end{proof}

        \begin{corollary}
            For $n_2=1$, the Signless Laplacian characteristics polynomial is the same as the Signless Laplacian characteristics polynomial of the corona product of two graphs.
        \end{corollary}
        \begin{proof}
            Proof is same as Corollary~\ref{C1}.
        \end{proof}

        \begin{corollary}
         Let $\Sigma_1$ and $\Sigma_2$ be two $Q(\Sigma_{\mu})$-co-spectral signed graphs and $\Sigma$ is any arbitrary signed graph, then 
            \begin{enumerate}[label=(\roman*)]
                \item $\Sigma_1\circledast\Sigma$ and $\Sigma_2\circledast\Sigma$ are $Q$-co-spectral.
                \item $\Sigma\circledast\Sigma_1$ and $\Sigma\circledast\Sigma_2$ are $Q$-co-spectral if $\chi_{Q(\Sigma_{1\mu})}(x)=\chi_{Q(\Sigma_{2\mu})}(x).$
            \end{enumerate}
     \end{corollary}

    \section{Application}
    In this section we discuss some application of signed graph product $\Sigma_1\circledast\Sigma_2$ and its spectrum in generating infinite family of integral signed graphs, and sequence of non-co-spectral equienergetic signed graphs.
    \subsection{Integral Signed Graphs}
    The following Theorem shows the conditions for $\Sigma_1\circledast\Sigma_2$ to be an integral signed graph.

    \begin{theorem}
        Let $i=1,2$, $\Sigma_i$ be two signed graph with $n_i$ vertices. Then $\Sigma_1\circledast\Sigma_2$ is integral signed graph if and only if the roots of $R'_{n_2-d}(x)$ and for $i=1,2,\cdots,n_1$ $x{F'}_d(x)-n_2({\lambda'}_{1i}{F'}_d(x)+{P'}_{d-1}(x))$ are integers.
    \end{theorem}
    \begin{proof}
        Proof follows from Theorem~\ref{Th2}.
    \end{proof}
    We us the result by Bishal et. al.~\cite{sonar2023spectrum}, let $\Sigma=(K_{1,n},\sigma,\mu)$ be a signed star with $V(\Sigma)=\{u_1,u_2,\cdots,u_{n+1}\}$ such that $d(u_1)=n$, then 
    $$\chi_{A(\Sigma)}(x)=\frac{(n+1)x+2n\mu(u_1)}{x^2-n}$$.
    \begin{prop}
        Let $\Sigma_1$ be a signed graph with $n_1$ vertices and $\Sigma_2=(K_{1,n},\sigma_2,\mu_2)$ be a signed star with $n_2=n+1$ vertices, then the product $\Sigma_1\circledast\Sigma_2$ will be integral if the underlying graph $S_2$ of $\Sigma_2$ has integral eigenvalue and the cubic equation $x^3-n_2{\lambda'}_{1i}x^2-(n_2^2-n_2+1)x+n_2(n_2-1)({\lambda'}_{1i}-2\mu_2(u_1))$ has integral roots for each $i=1,2,\cdots,n_1$. 
    \end{prop}
    \begin{proof}
        Here, $\Sigma_2$ is a signed star and star being a tree is balanced. So the eigenvalues of $A(\Sigma_{2\mu})$ is same as the eigenvalues of $A(\Sigma_2)$ and  $\chi_{A(\Sigma_{2\mu})}(x)=\chi_{A(\Sigma_2)}(x)=\frac{(n+1)x+2n\mu(u_1)}{x^2-n}$.\\
        From equation(\ref{Eqn6}) we have 
        \begin{equation*}
            \begin{split}
                f(A(\Sigma_1\circledast\Sigma_2),x)&=x^{n_1(n_2-1)}f(A(\Sigma_{2\mu}),x)^{n_1}\cdot \prod_{i=1}^{n_1}\big[x-n_2({\lambda}_{1i}+\chi_{A(\Sigma_{2\mu})}(x))\big]\\
                &=x^{n_1(n_2-1)}f(A(\Sigma_2),x)^{n_1}\cdot \prod_{i=1}^{n_1}\big[x-n_2({\lambda'}_{1i}+\chi_{A(\Sigma_2)}(x))\big].
            \end{split}
        \end{equation*}
        Clearly, the above function will have integral roots, if the roots of $f(A(\Sigma_2),x)$ and the equation $x-n_2({\lambda'}_{1i}+\chi_{A(\Sigma_2)}(x))$ for $i=1,2,\cdots,n_1$ are integers.\\
        Now, for $i=1,2,\cdots,n_1$,
        \begin{equation*}
                \begin{split}
                    &x-n_2({\lambda}_{1i}+\chi_{A(\Sigma_2)}(x))=0\\
                    \implies &x-n_2\Big({\lambda}_{1i}+\frac{(n+1)x+2n\mu_2(u_1)}{x^2-n}\Big)=0\\
                    \implies&x^3-n_2{\lambda}_{1i}x^2-(n_2^2+n_2-1)x+n_2(n_2-1)({\lambda}_{1i}-2\mu_2(u_1))=0.
                \end{split}
        \end{equation*}
    \end{proof}

    \subsection{Equienergetic signed graphs}
    In the following Theorem~\ref{TH4} we describe a method for constructing non-co-spectral equienergetic signed graphs.

    \begin{theorem}~\label{TH4}
        Let $\Sigma_1$ and $\Sigma_2$ be two signed graphs and $\Sigma_{1\mu}$ and $\Sigma_{2\mu}$ be the corresponding $\mu-$signed graph which are non-co-spectral equienergetic of order $m$ with same coronal, then for any arbitrary signed graph $\Sigma$ of order $n$, the signed graphs $\Sigma\circledast\Sigma_1$ and $\Sigma\circledast\Sigma_2$ are also non-co-spectral equienergetic.
    \end{theorem}
    \begin{proof}
        Since $\Sigma_1$ and $\Sigma_2$ are having same coronal, the polynomials $P_{d-1}(x)$ and $F_d(x)$ for both $\Sigma_1$ and $\Sigma_2$ given by equation~(\ref{Eqn5}) are equal.\\ 
        Let \begin{equation}
                f(A(\Sigma_1),x)={R'}^1_{m-d}(x)F'_d(x)
        \end{equation}
        and\begin{equation}
                f(A(\Sigma_2),x)={R'}^2_{m-d}(x)F'_d(x)
        \end{equation}
        Clearly, ${R'}^1_{m-d}\neq{R'}^2_{m-d}$, as the graphs $\Sigma_{1\mu}$ and $\Sigma_{2\mu}$ are non-co-spectral. Then the characteristics polynomial of $A(\Sigma\circledast\Sigma_1)$ is,
        \begin{equation}\label{Eqn10}
            f(A(\Sigma\circledast\Sigma_1),x)=x^{n(m-1)}{R'}^1_{m-d}(x)^{n}\prod_{i=1}^{n}[x{F'}_d(x)-m({\lambda'}_{1i}{F'}_d(x)+{P'}_{d-1}(x))]
        \end{equation}
        Similarly the characteristics polynomial of $A(\Sigma\circledast\Sigma_2)$ is,
        \begin{equation} \label{Eqn11}
           f(A(\Sigma\circledast\Sigma_1),x)=x^{n(m-1)}{R'}^2_{m-d}(x)^{n}\prod_{i=1}^{n}[x{F'}_d(x)-m({\lambda'}_{1i}{F'}_d(x)+{P'}_{d-1}(x))]
        \end{equation}
        Let the roots of the polynomial $F'_d(\lambda)$ be $\delta_1,\delta_2,\cdots,\delta_d$ and the roots of ${R'}^1_{m-d}(\lambda)$ and ${R'}^2_{m-d}(\lambda)$ be $\alpha_1,\alpha_2,\cdots,\alpha_{m-d}$ and $\beta_1,\beta_2,\cdots,\beta_{m-d}$ respectively.\\
        Now since the signed graphs $\Sigma_{1\mu}$ and $\Sigma_{2\mu}$ are equienergetic, we have,
        \begin{equation*}
            \sum_{i=1}^d|\delta_i|+\sum_{i=1}^{m-d}|\alpha_i|=E_{\Sigma_{1\mu}}=E_{\Sigma_{2\mu}}=\sum_{i=1}^d|\delta_i|+\sum_{i=1}^{m-d}|\beta_i|
        \end{equation*}
        Hence,\begin{equation}\label{Eqn12}
            \sum_{i=1}^{m-d}|\alpha_i|=\sum_{i=1}^{m-d}|\beta_i|.
        \end{equation}
        The product factor in equations (\ref{Eqn10}) and (\ref{Eqn11}) are same and its an $n(d+1)$ degree polynomial. Let its roots be $\eta_1,\eta_2,\cdots,\eta_{n(d+1)}$. From the characteristics polynomial (\ref{Eqn10}), the energy of the signed graph $\Sigma\circledast\Sigma_1$ is, \begin{equation}\label{Eqn13}
                E_{\Sigma\circledast\Sigma_1}=n\sum_{i=1}^{m-d}|\alpha_i|+\sum_{i=1}^{n(d+1)}|\eta_i|
        \end{equation}
        and from the characteristics polynomial (\ref{Eqn11}), the energy of the signed graph $\Sigma\circledast\Sigma_2$ is,
        \begin{equation}\label{Eqn14}
            E_{\Sigma\circledast\Sigma_2}=n\sum_{i=1}^{m-d}|\beta_i|+\sum_{i=1}^{n(d+1)}|\eta_i|
        \end{equation}
        Then from the above two equations and equation~(\ref{Eqn12}), we have $E_{\Sigma\circledast\Sigma_1}=E_{\Sigma\circledast\Sigma_2}$.\\
        As $\Sigma_{1\mu}$ and $\Sigma_{2\mu}$ are non-co-spectral signed graphs, so $\Sigma\circledast\Sigma_1$ and $\Sigma\circledast\Sigma_2$ are non-co-spectral equienergetic.
   \end{proof}
   The above theorem can be used to generate families of non-co-spectral equienergetic signed graphs from a given pair of non-co-spectral equienergetic $\mu$-signed graphs with same coronals. The following corollary gives the existence of such non trivial non-co-spectral equienergetic signed graphs with same coronals.

    \begin{corollary}
        Let $\Sigma_1$ and $\Sigma_2$ be two non-co-spectral equienergetic $r$-regular signed graphs, then for any arbitrary signed graph $\Sigma$, the signed graphs $\Sigma\circledast\Sigma_1$ and $\Sigma\circledast\Sigma_2$ are also non-co-spectral equienergetic.
    \end{corollary}
    \begin{proof}
        Suppose that $\Sigma_1$ and $\Sigma_2$ are non-co-spectral equienergetic $r$-regular signed graphs. It is known that the coronal of any two $r$-regular signed graphs are equal. Hence by Theorem~\ref{TH4}, the signed graphs $\Sigma\circledast\Sigma_1$ and $\Sigma\circledast\Sigma_2$ are non-co-spectral energetic signed graphs.
    \end{proof}
    
    \section*{Acknowledgement}
        We would like to acknowledge National Institute of Technology Sikkim for giving doctoral fellowship to Bishal Sonar.
    \section*{Funding}
        The authors declare that no funds, grants, or other support were received during the preparation of this manuscript.
    \section*{Data Availability}
        Data sharing not applicable to this article as no datasets were generated or analyzed during the current study.
    \section*{Deceleration}
    \subsection*{Conflict of interest}
        The authors have no relevant financial or non-financial interests to disclose.

    \bibliographystyle{abbrv}
    \bibliography{main.bib}
\end{document}